\documentclass[10pt]{amsart}
\usepackage{enumerate,amsmath,amssymb,latexsym,
amsfonts, amsthm, amscd}

\theoremstyle{plain}

\newtheorem{theorem}{Theorem}[section]

\numberwithin{equation}{section}

\newcommand{\ra}{\rightarrow}

\newcommand{\st}{\overline{\otimes}}
\newcommand{\E}{\mathbb{E}}

\addtocounter{section}{-1}

\begin{document}

\title {`Twisted Duality' for Clifford Algebras}

\date{}

\author[P.L. Robinson]{P.L. Robinson}

\address{Department of Mathematics \\ University of Florida \\ Gainesville FL 32611  USA }

\email[]{paulr@ufl.edu}

 \keywords{}

\begin{abstract}

Let $V$ be a real inner product space and $C(V_{\mathbb{C}})$ the Clifford algebra of its complexification $V_{\mathbb{C}}$. We present several proofs of the fact that if $W$ is a subspace of $V_{\mathbb{C}}$ then  $C(W^{\perp})$ coincides with the supercommutant of $C(W)$ in $C(V_{\mathbb{C}})$. 

\end{abstract}

\maketitle

\bigbreak

\section{Introduction}

Let $V$ be a real vector space upon which $(\cdot | \cdot )$ is a positive-definite inner product. This inner product extends to the complexification $V_{\mathbb{C}}$ to define both a nonsingular symmetric complex-bilinear form (for which we use the same symbol) and a Hermitian inner product $\langle \cdot | \cdot \rangle$; these forms are related by the identity $\langle x | y \rangle = ( \overline{x} | y)$ whenever $x, y \in V_{\mathbb{C}}$ where the overline signifies complex conjugation pointwise fixing $V \subseteq V_{\mathbb{C}}$. We shall denote the $(\cdot | \cdot )$-orthogonal space to the subspace $W \leqslant V_{\mathbb{C}}$ by 
\[W^{\perp} = \{ z \in V_{\mathbb{C}} | (\forall w \in W) \; (w | z ) = 0 \}
\]
and use the same notation for subspaces of $V$ itself. Note that if $W \leqslant V_{\mathbb{C}}$ then $W^{\perp} \cap W$ need not be zero (indeed,  $W^{\perp} = W$ is possible) whereas if $Z \leqslant V$ then $Z^{\perp} \cap Z = 0$. 

\medbreak

Recall that the Clifford algebra of $V$ is a unital associative real algebra $C(V)$ with a preferred linear embedding $V \rightarrow C(V)$ that satisfies the Clifford relation 
\[ (\forall v \in V) \; \; v v = (v | v) \bf{1}
\]
and has the following universal mapping property (UMP): if $V \rightarrow A$ linearly embeds $V$ in the unital associative algebra $A$ and satisfies the Clifford relation, then there exists a unique algebra homomorphism $C(V) \rightarrow A$ that restricts to the identity on $V$. The Clifford algebra of a complex vector space equipped with a symmetric bilinear form is the unital associative complex algebra defined similarly; note that $C(V_{\mathbb{C}})$ may be identified with the complexification $C(V)_{\mathbb{C}}$. 

\medbreak 

Clifford algebras naturally carry important structural maps. First among these, the linear operator $- {\rm Id}$ on $V$ extends uniquely (via the UMP) to an algebra automorphism of $C(V)$ which we denote by $\gamma$ and call the grading automorphism; its fixed points constitute the even Clifford algebra $C (V)_+$ while the fixed points of $- \gamma$ constitute the odd subspace $C(V)_-$. In like fashion, the complex Clifford algebra $C(V_{\mathbb{C}})$ is similarly graded by an analogous grading automorphism. The importance of the grading automorphism $\gamma$ is that it makes $C(V_{\mathbb{C}})$ into a superalgebra, which significantly clarifies its structure; in particular, we remark without proof that  $C(V_{\mathbb{C}})$ is simple as a superalgebra. Next in importance, the complex Clifford algebra $C(V_{\mathbb{C}})$ carries a unique involution $^*$ that restricts to $V_{\mathbb{C}} \subseteq C(V_{\mathbb{C}})$ as complex conjugation pointwise fixing $V$: to see this, apply the UMP to the embedding of $V_{\mathbb{C}}$ in the algebra obtained from $C(V_{\mathbb{C}})$ by conjugating the linear structure and reversing the product; thus, $C(V_{\mathbb{C}})$ is naturally a unital associative $^*$-algebra. Of course, this involution restricts to $C(V) \subseteq C(V_{\mathbb{C}})$ as a period-two antiautomorphism. Finally,  $C(V_{\mathbb{C}})$ carries a unique trace: that is, a $\gamma$-invariant linear functional $\tau :  C(V_{\mathbb{C}}) \rightarrow \mathbb{C}$ such that $\tau ({\bf 1}) = 1$ and such that $\tau (a b) = \tau (b a)$ whenever $a, b \in  C(V_{\mathbb{C}})$. A Hermitian inner product $\langle \cdot | \cdot \rangle_{\tau}$ is then defined on $C(V_{\mathbb{C}})$ by the rule 
\[ (\forall a, b \in C(V_{\mathbb{C}})) \; \; \langle a | b\rangle_{\tau} = \tau (a^* b). 
\]

\medbreak 

Now, let $W \leqslant V_{\mathbb{C}}$ be a complex subspace of the complexification. Its Clifford algebra is naturally graded, thus: 
\[ C(W) = C(W)_+ \oplus C(W)_-. 
\]
According to the Koszul-Quillen rule of signs, the {\it supercommutant} of $C(W)$ in $C(V_{\mathbb{C}})$ is the subalgebra 
\[C(W)' = C(W)'_+ \oplus C(W)'_-
\]
with even part 
\[ C(W)'_+ = \{ a \in C(V_{\mathbb{C}})_+ | (\forall b \in C(W) ) \; b a = a b\}
\]
and odd part 
\[ C(W)'_- = \{ a \in C(V_{\mathbb{C}})_- | (\forall b \in C(W) ) \; b a =a \gamma(b)\}. 
\]
\medbreak
\noindent
Our aim in this paper is to collect together several proofs of the following theorem, which is an abstract formulation of `twisted duality'. 

\begin{theorem} 
If $W \leqslant V_{\mathbb{C}}$ is a subspace of the complexification, then its Clifford algebra has supercommutant in  $C(V_{\mathbb{C}})$ given by 
\[ C(W)' = C(W^{\perp}). 
\]
\end{theorem}

Notice that we have made no assumption regarding dimension; this theorem is valid in both finite and infinite dimensions. Having said this, we will find it convenient to work up through finite dimensions. Explicitly, let $\mathcal{F} (V)$ denote the set of finite-dimensional subspaces of $V$ directed by inclusion: then 
\[ C(V) = \bigcup_{M \in \mathcal{F} (V)} C(M);
\]
likewise
\[ C(V_{\mathbb{C}}) = \bigcup_{N \in \mathcal{F} (V_{\mathbb{C}})} C(N).
\]
\medbreak 

In Section 1 we shall present a proof of Theorem 0.1 in full generality; indeed, our proof turns out to be valid over arbitrary scalar fields. In Section 2 we offer a rather different proof making use of a (super) tensor product decomposition, which is applicable to orthogonal (direct) decompositions of $V_{\mathbb{C}}$. In Section 3 we offer yet another proof that develops and makes use of conditional expectations; this proof applies to orthogonal decompositions of $V$ itself. 

\medbreak

`Twisted duality' was introduced in [2] and was developed in [5]; abstract `twisted duality' was studied in [3] (based on a 1982 thesis) and more recently in [1]. Our discussion of conditional expectations in Section 3 extends and simplifies the approach in [4], which serves as a convenient reference for the theory of Clifford algebras. We remark that among the clarifications that arise from viewing the Clifford algebra as a {\it superalgebra} is clarification of `twisted duality' itself: the Koszul-Quillen rule obviates the need for (or perhaps subsumes) the Klein transformation in terms of which `twisted duality' is usually formulated. 

\section{First Proof: General Case}

Fix a complex subspace $W \leqslant V_{\mathbb{C}}$ and consider the supercommutant
\[  C(W)' = C(W)'_+ \oplus C(W)'_-. 
\]
Its even part comprises precisely all those $a \in C(V_{\mathbb{C}})_+$ such that 
\[ (\forall b \in C(W)) \; b a = a b
\]
equivalently such that 
\[ (\forall w \in W) \; w a = a w = \gamma(a) w,
\]
while its odd part comprises precisely all those $a \in C(V_{\mathbb{C}})_-$ such that 
\[ (\forall b \in C(W)) \; b a = a \gamma(b)
\]
equivalently such that 
\[ (\forall w \in W) \; w a = a (-w) = \gamma(a) w. 
\]
Thus 
\[ C(W)' = \{ a \in C(V_{\mathbb{C}}) | (\forall w \in W) \; wa = \gamma(a) w \} . 
\]

\medbreak 

Notice at once that the inclusion
\[ C(W^{\perp}) \subseteq C(W)'
\]
follows immediately from the Clifford relations; accordingly, we shall only need to establish the reverse inclusion 
\[ C(W)' \subseteq C(W^{\perp}). 
\]

\medbreak 

We begin by considering first the case in which $W \leqslant V_{\mathbb{C}}$ is one-dimensional: say $W = \mathbb{C} w$ for some nonzero $w \in W$. Note that $\overline{w} \notin W^{\perp}$: in fact, $( \overline{w} | w) = \langle w | w \rangle > 0$; it follows that $V_{\mathbb{C}} = W^{\perp} \oplus  \mathbb{C} \overline{w}$ is the direct sum of the complex  hyperplane $W^{\perp}$ and the complex line $\mathbb{C} \overline{w}$.

\begin{theorem} 
\[C(V_{\mathbb{C}}) = C(W^{\perp}) \oplus \overline{w} C(W^{\perp}). 
\]
\end{theorem}

\begin{proof} 
To prove that $ C(V_{\mathbb{C}})$ is the sum of the two spaces on the right of the alleged equation, we show that each Clifford product $v_0 v_1 \cdots v_N \in  C(V_{\mathbb{C}})$ of vectors in $V_{\mathbb{C}}$ lies in that sum. The base step is clear: each $v_0 \in V_{\mathbb{C}}$ decomposes as 
\[ v_0 = u(v_0) + \lambda(v_0)  \overline{w} 
\]
for unique $u(v_0) \in W^{\perp}$ and $ \lambda(v_0)  \in \mathbb{C}$. For the inductive step, write 
\[ v_0 v_1 \cdots v_N = (u(v_0) + \lambda(v_0)  \overline{w})v_1 \cdots v_N
\] 
where (inductively) 
\[ v_1 \cdots v_N = a + \overline{w} b
\]
with $a, b \in C(W^{\perp})$. In the subsequent expansion, note that by the (linearized) Clifford relations, $ \overline{w}  \overline{w} = \overline{(w | w)}$  and $u(v_0) \overline{w} + \overline{w} u(v_0) = 2 (\overline{w} | u(v_0) ) = 2 \langle w | u(v_0) \rangle$. It follows that 
\[ v_0 v_1 \cdots v_N = A + \overline{w} B
\] 
where 
\[ A = u(v_0) a + (\lambda(v_0)  \overline{(w | w)} + 2  \langle w | u(v_0) \rangle) b
\]
and 
\[ B = \lambda(v_0) a - u(v_0) b
\]
both lie in $C(W^{\perp})$. \par 
To prove that the sum is direct, let  $a, b \in C(W^{\perp})$ satisfy $a + \overline{w} b = 0$: then $w \overline{w} b = - w a = - \gamma(a) w = \gamma(\overline{w} b) w = - \overline{w} \gamma(b) w = - \overline{w} w b$ where the first and third equalities hold by assumption on $a$ and $b$, the second and fifth because $C(W^{\perp}) \subseteq C(W)'$ and the fourth because $\overline{w}$ is odd; thus $2\langle w | w \rangle b = (\overline{w} w + w \overline{w})b = 0$ and so $b = 0$. 
\end{proof}

Now we prove Theorem 0.1 in case $W = \mathbb{C} w$. Let $c \in C(W)'$: thus, $w c = \gamma(c) w$ . Use Theorem 1.1 to express $c$ uniquely as 
\[ c = a + \overline{w} b
\]
with $a, b \in C(W^{\perp})$. Now $w c = \gamma(c) w$ reads
\[ w a + w \overline{w} b = \gamma(a) w - \overline{w} \gamma(b) w = w a - \overline{w} w b
\]
so that $2 \langle w | w \rangle b =  (\overline{w} w + w \overline{w})b = 0$ and therefore $b = 0$; this places $c = a$ in $C(W^{\perp})$ as required. Having now established Theorem 1.1 in case $W = \mathbb{C} w$ we may express it in the form 
\[ C(w)' = C(w^{\perp}).
\]
\medbreak 

In order to complete the proof of Theorem 0.1 we now address the intersection of Clifford algebras, beginning with finite intersections. 

\begin{theorem} 
If $X$ and $Y$ are subspaces of $ V_{\mathbb{C}}$ then $C(X) \cap C(Y) = C(X \cap Y)$. 
\end{theorem} 

\begin{proof} 
Only the inclusion $C(X) \cap C(Y) \subseteq C(X \cap Y)$ need be checked. Let $c \in C(X) \cap C(Y)$ : choose $M \in \mathcal{F} (X)$ and $N \in \mathcal{F} (Y)$ so that $c \in C(M) \cap C(N)$; if we can prove that $c \in C(M \cap N)$ then we shall be done. Thus, we may and shall assume without loss that $X$ and $Y$ are finite-dimensional. Let $Z : = X \cap Y$: choose complements $X_Z$ and $Y_Z$ so that $X = Z \oplus X_Z$ and $Y = Z \oplus Y_Z$; choose bases $\{x_i \}$, $\{ y_j \}$, $\{ z_k \}$ for $X, Y, Z$ respectively. According to standard multi-index notation, if $K = (k_1, \dots, k_r)$ is a sequence of integers with $1 \leqslant k_1 < \dots < k_r \leqslant \dim Z$ then $z_K = z_{k_1} \cdots z_{k_r}$ denotes the Clifford product, while if $K$  is the empty sequence then $z_K = {\bf 1}$; interpret $x_I$ and $y_J$ for strictly increasing multi-indices $I$ and $J$ in like manner. Note that the products $x_I y_J z_K$ form a basis for $C(X + Y)$ while the products $z_K$ form a basis for $C(Z)$ and so on. In these terms, write 
\[ c = \sum_{I, J, K} \lambda_{IJK} x_I y_J z_K. 
\]
As $c$ lies in $C(X)$ it follows that if $J \neq \emptyset$ then $\lambda_{IJK} = 0$; as $c$ lies in $C(Y)$ it follows that  if $I \neq \emptyset$ then $\lambda_{IJK} = 0$. Thus 
\[  c = \sum_{K} \lambda_{\emptyset \emptyset K} x_{\emptyset} y_{\emptyset} z_K = \sum_K \lambda_{\emptyset \emptyset K} z_K \in C(Z).
\]
\end{proof}

By induction, we conclude that this result for pairwise intersections extends to finite intersections; it actually extends to arbitrary intersections, as follows. 

\begin{theorem} 
If $\{ Z_{\lambda} | \lambda \in \Lambda \}$ is any family of subspaces of $V_{\mathbb{C}}$ then 
\[ \cap_{\lambda} C(Z_{\lambda}) = C(\cap_{\lambda} Z_{\lambda}).
\]
\end{theorem}

\begin{proof} 
It is enough to show that both sides have the same intersection with $C(N)$ for each $N \in \mathcal{F} (V_{\mathbb{C}})$ on acount of Theorem 1.2 and the fact that these Clifford algebras have $C(V_{\mathbb{C}})$ as their union. Accordingly, we may and shall assume without loss that $V$ is finite-dimensional and need only establish the inclusion 
\[ \cap_{\lambda} C(Z_{\lambda}) \subseteq C(\cap_{\lambda} Z_{\lambda}).
\]
Write $\overline{\Lambda}$ for the collection comprising all finite subsets of $\Lambda$; when $F \in \overline{\Lambda}$ write $Z_F$ for the (finite!) intersection of $Z_{\lambda}$ as $\lambda$ runs over $F$. Let $F_0 \in \overline{\Lambda}$ be such that $Z_{F_0}$ has least dimension among the (finite-dimensional!) subspaces $\{ Z_F : F \in \overline{\Lambda} \}$. If $\lambda \in \Lambda$ then on the one hand $Z_{F_0 \cup \{ \lambda \}} = Z_{F_0} \cap Z_{\lambda} \subseteq Z_{F_0}$ and on the other hand $\dim Z_{F_0 \cup \{ \lambda \}} \geqslant \dim Z_{F_0}$ by minimality; thus $Z_{F_0} \cap Z_{\lambda} = Z_{F_0}$ and so $Z_{F_0} \subseteq  Z_{\lambda}$. This proves 
\[ \cap_{\lambda} Z_{\lambda} = Z_{F_0}. 
\]
Finally, Theorem 1.2 justifies the middle step in 
\[ \cap_{\lambda} C(Z_{\lambda}) \subseteq \cap_{\lambda \in F_0} C(Z_{\lambda}) = C(Z_{F_0}) = C(\cap_{\lambda} Z_{\lambda} ). 
\]
\end{proof}

We are now able to prove Theorem 0.1 in full generality: if $W \leqslant  V_{\mathbb{C}}$ then 
\[ C(W)' = \cap_{w \in W} C(w)' = \cap_{w \in W} C(w^{\perp})
\]
as noted after Theorem 1.1; now Theorem 1.3 yields 
\[ \cap_{w \in W} C(w^{\perp}) = C(\cap_{w \in W} w^{\perp}) = C(W^{\perp}).
\]

\medbreak 
We observe that this proof actually works for nonsingular symmetric bilinear forms over arbitrary scalar fields (of characteristic other than two) if in Theorem 1.1 the vector $\overline{w}$ is replaced by any vector not in $W^{\perp}$. 

\medbreak 

As a special case, we recover the following familiar fact. 

\begin{theorem} 
$C(V_{\mathbb{C}})$ has scalar supercentre. 
\end{theorem}

\begin{proof} 
Simply note that $C(V_{\mathbb{C}})' = C(0) = \mathbb{C} {\bf 1}$. 
\end{proof}  

Of course, this may be proved directly: if $a \in C(V_{\mathbb{C}})'$ is given, then choose $M \in \mathcal{F} (V)$ so that $a \in C(M_{\mathbb{C}})$; expand $a$ in terms of an orthonormal basis for $M$ and invoke $v a = \gamma(a) v$ for each basis vector $v$.  

\section{Second Proof: Tensor Products}

For the setting of our second proof, we assume an orthogonal direct sum decomposition 
\[ V_{\mathbb{C}} = X \oplus Y
\]
into complex subspaces $X$ and $Y$ of $V_{\mathbb{C}}$. Thus: not only do we assume that $Y = X^{\perp}$ and $X = Y^{\perp}$; we also assume that $X \cap Y = 0$. Note that if $W \leqslant V_{\mathbb{C}}$ is an arbitrary subspace then it need not be the case that $W$ and $W^{\perp}$ are complementary: on the one hand, $W \cap W^{\perp}$ can be nonzero; on the other hand, even when $W \cap W^{\perp} = 0$ it need not be the case that $W + W^{\perp} = V_{\mathbb{C}}$. Of course, this means that we shall actually offer here a proof of a somewhat weaker result than Theorem 0.1: namely, that 
\[ C(X)' = C(Y).
\]

\medbreak 

We begin by recalling the super tensor product $\st$ in this Clifford algebra context. The super tensor product $C(X) \st C(Y)$ is the superalgebra with the ordinary tensor product $C(X) \otimes C(Y)$ as underlying vector space but with multiplication given on homogeneous elementary tensors by the Koszul-Quillen rule 
\[ (a_1 \st b_1) (a_2 \st b_2) = (-1)^{\partial(b_1) \partial(a_2)}(a_1 a_2) \st (b_1 b_2)
\]
where $a_1, a_2 \in C(X)$ and $b_1, b_2 \in C(Y)$ and where the degree $\partial$ is $0$ on even elements and $1$ on odd elements; in particular, if also $y_1 \in Y$ then $\partial(y_1) = 1$ so that 
\[ (a_1 \st y_1) (a_2 \st b_2) = (a_1 \gamma(a_2)) \st (y_1 b_2). 
\]. 

Now, consider the complex-linear map from $V_{\mathbb{C}} = X \oplus Y$ to $C(X) \st C(Y)$ given by 
\[ \phi: X \oplus Y \rightarrow C(X) \st C(Y) : x \oplus y \mapsto x \st {\bf 1} + {\bf 1} \st y.
\]
In view of  the Koszul-Quillen rule, $ ({\bf 1} \st y) (x \st {\bf 1}) = \gamma(x) \st y = - x \st y = - (x \st {\bf 1}) ({\bf 1} \st y)$; thus cross-terms cancel  when $\phi(x \oplus y)$ is squared and so $\phi$ satisfies the Clifford relation: 
\[ \phi(x \oplus y)^2 = (x^2) \st {\bf 1} + {\bf 1} \st (y^2) = ((x | x) + (y | y)) {\bf 1} = (x \oplus y | x \oplus y) {\bf 1}.
\]
The UMP extends $\phi$ to a superalgebra homomorphism 
\[ \Phi:  C(V_{\mathbb{C}}) \rightarrow C(X) \st C(Y) 
\]
that is actually an isomorphism: injective because the superalgebra $ C(V_{\mathbb{C}})$ is simple; surjective because it restricts to $C(X) \leqslant  C(V_{\mathbb{C}})$ as $a \mapsto a \st {\bf 1}$ and to $C(Y) \leqslant  C(V_{\mathbb{C}})$ as $b \mapsto {\bf 1} \st b$. We shall feel free to identify $C(V_{\mathbb{C}})$ with $C(X) \st C(Y)$ via this canonical isomorphism. Note that the grading automorphism $\gamma$ of $C(V_{\mathbb{C}}) \equiv C(X) \st C(Y)$  maps $a \st b$ to $\gamma(a) \st \gamma(b)$. 

\medbreak

In these terms, our version of Theorem 0.1 is the following claim: 
\[ (C(X) \st \mathbb{C} {\bf 1})' = \mathbb{C} {\bf 1} \st C(Y)
\]
which we now justify as follows. Let $c \in (C(X) \st \mathbb{C} {\bf 1})'$. As an element of the full tensor product, $c$ has a decomposition 
\[ c = \sum_{n = 1}^N a_n \st b_n
\]
with $\{ a_1, \dots , a_N \} \subseteq C(X)$ and with linearly independent $\{ b_1, \dots , b_N \} \subseteq C(Y)$. Membership of $c$ in the supercommutant $(C(X) \st \mathbb{C} {\bf 1})'$ is equivalent to each of the following for all $x \in X$: 
\[ (x \st {\bf 1}) c = \gamma(c) (x \st {\bf 1})
\]
\[ \sum_{n = 1}^N  (x \st {\bf 1}) (a_n \st b_n) = \sum_{n = 1}^N (\gamma(a_n) \st \gamma(b_n))(x \st {\bf 1})
\]
\[ \sum_{n = 1}^N  x a_n \st b_n = \sum_{n = 1}^N \gamma(a_n) x \st b_n
\]
using the Koszul-Quillen rule at the last step. As the vectors $\{ b_1, \dots , b_N \}$ are linearly independent, we deduce that $c \in (C(X) \st \mathbb{C} {\bf 1})'$ is equivalent to 
\[ (\forall n \in \{ 1, \dots , N \}) (\forall x \in X) \; \; x a_n = \gamma(a_n) x
\]
hence (see Theorem 1.4) to 
\[  (\forall n \in \{ 1, \dots , N \}) \; \; a_n = \alpha_n {\bf 1} \in \mathbb{C} {\bf 1}
\]
whence 
\[ c = \sum_{n = 1}^N \alpha_n {\bf 1} \st b_n =  \sum_{n = 1}^N {\bf 1} \st \alpha_n b_n \in \mathbb{C} {\bf 1} \st C(Y)
\]
as claimed. 

\medbreak 

Observe that once again, our proof actually works for nonsingular symmetric bilinear forms over arbitrary fields of characteristic other than two. 

\medbreak 

We should perhaps close this Section by expanding upon comments we made at the opening. Let $W \leqslant V_{\mathbb{C}}$ be a subspace. On the one hand, $W \cap W^{\perp}$ might be nonzero: for example, if $J : V \rightarrow V$ is an orthogonal transformation with square $- {\rm Id}$ then $W : = \{ v - i J v | v \in V \}$ satisfies $W^{\perp} = W$; in this example, the real dimension of $V$ is other than odd. On the other hand, even when $W \cap W^{\perp}$ is zero, the sum $W + W^{\perp}$ might fall short of $V_{\mathbb{C}}$: for example, $V$ might be a real Hilbert space and $W = Z_{\mathbb{C}}$ for some subspace $Z \leqslant V$ that is not closed; in this case, the dimension of $W$ is infinite. 

\section{Third Proof: Conditional Expectations}

For the setting of our third proof, we assume an orthogonal direct sum decomposition 
\[ V = X \oplus Y
\]
of $V$ itself. The idea is to construct a conditional expectation 
\[ \E_X : C(V_{\mathbb{C}}) \rightarrow C(Y_{\mathbb{C}}) 
\]
that acts on $C(X_{\mathbb{C}})'$ as the identity, and thereby to establish 
\[ C(X_{\mathbb{C}})' = C(Y_{\mathbb{C}}). 
\] 
Once again, we actually prove not Theorem 0.1 but a weaker variant. For convenience, we work with the real Clifford algebras and leave complexification for the reader: thus, we construct 
\[ \E_X : C(V) \rightarrow C(Y) 
\]
and use it to establish 
\[ C(X)' = C(Y). 
\]
As some of the proofs involve arguments that are closely similar to those already detailed in the preceding sections, we shall feel free to lighten our account. 

\medbreak 

To begin, let $u \in V$ be a unit vector. The direct sum decomposition 
\[C(V) = C(u^{\perp}) \oplus u C(u^{\perp}) 
\]
may be established as for Theorem 1.1 but more simply; in this case, the decomposition is orthogonal relative to the inner product $ \langle \cdot | \cdot\rangle_{\tau}$ on $C(V_{\mathbb{C}})  \supseteq C(V)$. By direct computation,
\[ C(u^{\perp}) = \{ a \in C(V) : u a u = \gamma(a) \} 
\]
and 
\[ u C(u^{\perp}) = \{ a \in C(V) : u a u = - \gamma(a) \}. 
\]
The $\langle \cdot | \cdot\rangle_{\tau}$-orthogonal projector of $C(V)$ on $C(u^{\perp})$ along $u C(u^{\perp})$ is thus given by 
\[ P_u : C(V) \ra C(V) : a \mapsto \frac{1}{2}(a + u \gamma(a) u).
\]

\medbreak 

Further, if $\{ u_1, \dots , u_m \}$ is an orthonormal set in $V$ with span $M \in \mathcal{F} (V)$ then the projectors $P_{u_1} , \dots , P_{u_m} $ commute; their product is the orthogonal projector on $C(u_1^{\perp}) \cap \cdots \cap C(u_m^{\perp}) = C(M^{\perp})$. We write this operator as 
\[ \E_M = P_{u_m} \circ \cdots \circ P_{u_1} : C(V) \ra C(M^{\perp}). 
\]
In fact, the orthogonal projector $\E_M$ is a conditional expectation. 

\begin{theorem} 
If $a \in C(V)$ and $b, c \in C(M)'$ then 
\[ \E_M (b a c) = b \E_M (a) c. 
\]
\end{theorem} 

\begin{proof} 
It will be enough to see that if $u \in M$ is a unit vector then $P_u$ has this property. To see this, note that $u \gamma(b) = bu$ and $\gamma(c) u = u c$ so that $u \gamma(b a c) u = b u \gamma(a) u c$ and therefore 
\[ 2 P_u (bac) = bac + u \gamma(b a c) u = bac + b u \gamma(a) u c = 2 b P_u (a) c
\]
as required. 
\end{proof}

Incidentally, notice that we have just established $C(M)' = C(M^{\perp})$. 

\medbreak 

We remark that $\E_M$ is $^*$-preserving and indeed positive: again we need only check $P_u$ and note that if $a \in C(V)$ then 
\[ 2 P_u (a^* a) = a^* a + (\gamma(a) u)^*  (\gamma(a) u)
\]
whence $P_u(a^* a)$ is a convex combination of terms $b^* b$ for $b \in C(V)$; so the same is true of $\E_M (a^* a)$. 

\medbreak 

Having thus dealt with finite-dimensional subspaces we consider the orthogonal decomposition  
\[ V = X \oplus Y
\]
with which we started this section. As we shall see, the net $( \E_M | M \in \mathcal{F} (X) )$ of conditional expectations indexed by the directed set of finite-dimensional subspaces of $X$ converges pointwise; its limit will be the conditional expectation $\E_X$. 

\medbreak 

Let $a \in C(V)$ and choose $N \in \mathcal{F} (V)$ so that $a \in C(N)$. Let the $( \cdot | \cdot )$-orthogonal projections of $N$ on $X$ and $Y$ be $X_N \in \mathcal{F} (X)$ and $Y_N \in \mathcal{F} (Y)$ respectively. From $a \in C(N) \subseteq  C(X_N \oplus Y_N)$ it follows that 
\[ \E_{X_N} (a) \in C((X_N \oplus Y_N) \cap (X_N)^{\perp}) = C(Y_N) \subseteq C(Y) \subseteq C(X)'. 
\] 
Consequently, if $u \in X$ is any unit vector then $u \gamma(\E_{X_N} (a)) u  = \E_{X_N} (a)$ and therefore 
\[ P_u \E_{X_N} (a) = \frac{1}{2} (\E_{X_N} (a) + u \gamma(\E_{X_N} (a)) u) = \E_{X_N} (a).
\] 
We may now see that the net $( \E_M (a)| M \in \mathcal{F} (X) )$ stabilizes, as follows. 

\begin{theorem} 
Let $a \in C(V)$ and choose $N \in \mathcal{F} (V)$ so that $a \in C(N)$. If $M \in \mathcal{F} (X)$ contains $X_N$ then 
\[ \E_M (a) = \E_{X_N} (a). 
\]
\end{theorem} 

\begin{proof} 
All we need do is refer to the equation displayed prior to the theorem and take the product of the projectors $P_u$ as $u$ runs over an orthonormal basis for $M \cap (X_N)^{\perp}$. 
\end{proof} 

It follows that we may pass to the limit and define a (plainly linear) map
\[  \E_X : C(V) \ra C(Y)
\]
by the rule that if $a \in C(V)$ then 
\[ \E_X (a) = \E_{X_N} (a) 
\]
where $X_N$ is the orthogonal projection on $X$ of any $N \in \mathcal{F} (V)$ such that $a \in C(N)$. The map $\E_X$ pointwise fixes $C(X)'$: if $a \in C(X)'$ and if $u \in X$ is a unit vector then $u \gamma(a) u = a$ so that $P_u (a) = a$; now let $u$ run over an orthonormal basis for $X_N$ in the notation established for Theorem 3.2.

\medbreak 

At this point, note that we have already established the equality
\[ C(X)' = C(Y). 
\]
Explicitly, the Clifford relations again imply $C(Y) \subseteq C(X)'$ while $C(X)' \subseteq C(Y)$ follows at once from the fact that $\E_X : C(V) \ra C(Y)$ fixes $C(X)'$ pointwise. 

\medbreak 

Having come this far, we ought to record some properties of $\E_X$ that follow immediately from its construction as the limit of $\E_M$ as $M$ runs over  $\mathcal{F} (X)$. As $\E_X$ fixes $C(X)'$ pointwise, it is an idempotent. The map $\E_X$ is $^*$-preserving and indeed positive: again, if $a \in C(V)$ then $\E_X (a^* a)$ is a convex combination of terms $b^* b$ for $b \in C(V)$. Also, $\E_X$ has the conditional expectation property: if $a \in C(V)$ and $b, c \in C(X)'$ then 
\[ \E_X (b a c) = b \E_X (a) c
\]
as may be seen by choosing $N \in \mathcal{F} (V)$ so large that $a, b, c \in C(N)$ and passing to $X_N$ in the notation for Theorem 3.2. 

\medbreak 

There is much more to say concerning these conditional expectations; having presented enough to fashion yet another proof of `twisted duality' as was our intention, we shall postpone further discussion to a future article. 

\medbreak 
\noindent 
\begin{center}
REFERENCES
\end{center}
\medbreak 
\noindent
[1] H. Baumgartel, M. Jurke and F. Lledo, {\it Twisted duality of the CAR-Algebra}, J. Math. Phys. {\bf 43 (8)} (2002) 4158-4179. 
\medbreak 
\noindent 
[2] S. Doplicher, R. Haag and J. Roberts, {\it Fields, observables and gauge transformations. I}, Comm. Math. Phys. {\bf 13} (1969) 1-23. 
\medbreak 
\noindent
[3] J. J. Foit, {\it Abstract Twisted Duality for Quantum Free Fermi Fields}, Publ. RIMS, Kyoto Univ. {\bf 19} (1983) 729-741. 
\medbreak 
\noindent 
[4] R. J. Plymen and P. L. Robinson, {\it Spinors in Hilbert Space}, Cambridge Tracts in Mathematics {\bf 114} (1994). 
\medbreak 
\noindent 
[5] S.J. Summers, {\it Normal Product States for Fermions and Twisted Duality for CCR- and CAR-Type Algebras with Application to the Yukawa$_2$ Quantum Field Model}, Comm. Math. Phys. {\bf 86} (1982) 111-141. 

\end{document}